\documentclass[12pt]{article}
\usepackage{amsgen, amsmath, amsfonts, amsthm, amssymb,
enumerate,amscd}
\usepackage[all]{xy}
\usepackage{srcltx}

\usepackage{color}

\definecolor{blue}{rgb}{0,0,1}
\definecolor{red}{rgb}{1,0,0}
\definecolor{green}{rgb}{0,1,0}
\definecolor{purple}{rgb}{1,0,1}

\long\def\red#1\endred{\textcolor{red}{#1}}
\long\def\blue#1\endblue{\textcolor{blue}{#1}}
\long\def\purple#1\endpurple{\textcolor{purple}{#1}}
\long\def\green#1\endgreen{\textcolor{green}{#1}}

\newtheorem{defn}{Definition}[section]
\newtheorem{prop}[defn]{Proposition}

\newtheorem{thm}[defn]{Theorem}

\newcommand {\ZZ}{{\mathbb Z}}
\newcommand {\sgn}{{\textrm{sgn}}}
\newcommand {\EEE}{{\mathbf E}}

\newcommand {\C}{{\mathbb C}}

\newcommand {\G}{{\Gamma}}

\newcommand {\Q}{{\mathbb Q}}
\newcommand {\R}{{\mathbb R}}

\newcommand {\g}{{\gamma}}

\newcommand {\HH}{{\mathfrak H}}

\def\mod{\operatorname{mod}}

\def\Im{\operatorname{Im}}
\def\Re{\operatorname{Re}}

\begin{document}
\title{Kernels of L-functions and shifted 
convolutions}
\author{Nikolaos Diamantis (University of Nottingham)\\
nikolaos.diamantis@nottingham.ac.uk}
\date{}

\maketitle

\bigskip

\section{Introduction}

\newcommand\lijn{\bigskip\par\hrule\bigskip\par}

Let $f$ be a normalised cuspidal eigenform 
of even weight $k$: 
 $$f(z)=\sum_{n=1}^{\infty}a(n) e^{2 \pi i n z}$$
  for SL$_2(\ZZ)$ and let $L_f(s)$ denote its L-function. The nature of 
the values $L_f(n)$ for $n \ge k$ has long been of interest and is 
related to Beilinson's conjecture. Thanks to work by Beilinson and 
Deninger-Scholl it is known that
all these values are, up to a power of $1/\pi$, periods in the 
sense of Kontsevich-Zagier (see \cite{KZ} and the references therein). 
Since then, Brunault, Rogers-Zudilin and others have, in some cases, 
established 
{\it explicit} expressions of non-critical values as periods and 
determined other aspects of their algebraic structure.

In \cite{DOS2}, C. O'Sullivan and the author established a new 
characterisation of the field containing an arbitrary value of $L_f(s)$, 
even for $s$ non-integer. The precise statement is reviewed in the 
next section (Prop. \ref{Period}) but it can be summarised as stating 
that ratios of non vanishing L-values of $f$ belong to the field 
generated by the Fourier coefficients of $f$ and of certain double 
Eisenstein series. These double Eisenstein series were constructed 
as, in a sense, generalisations to non-integer indices of the usual 
Rankin-Cohen  brackets. To obtain our characterisation we analysed 
those double Eisenstein series by a method parallel to that of 
\cite{Z}. As in \cite{Z}, the key was the expression of
$L_f(s)L_f(w)$ ($s, w \in \C$) as a Petersson 
scalar product of $f$ against a kernel induced by the double Eisenstein 
series.

On the other hand, \cite{B} gives a real analytic kernel for a 
product of L-values (Cor. 11.12 of \cite{B}). In that work, Brown 
uses Manin's theory of iterated Shimura integrals \cite{M1, M2} in 
order to study multiple modular values. The real analytic kernel is an 
essential ingredient of Brown's proof that each value $L_f(n)$ is, (up to 
a 
power of $\pi i$) a multiple modular value for SL$_2(\ZZ).$ 

Here we show that, in combination with \cite{DOS2}, the kernel defined in 
\cite{B} can be used to give a characterisation of the field of 
non-critical values that involves recognisable objects. 
That was an aim 
we were not able to reach in \cite{DOS2} and the 
expression in terms of recognisable objects we obtain here is a 
step in that direction. The recognisable objects are shifted divisor sum 
Dirichlet series and they have previously been studied in a 
completely 
different context. 

In general, shifted convolution series are the focus of various lines of
investigation because of their natural appearance in moment and 
other analytic problems. Shifted convolutions involving divisor sums, in 
particular, are important for the binary additive divisor problem and 
other related questions. 

The specific shifted convolution appearing here 
turns out to be essentially the one recently studied by M. K{\i}ral in 
\cite{K}, 
and which, in its domain of initial convergence is given by
$$D_h(\alpha, \beta; s):=\sum_{\substack{n \in \mathbb N \\ n>h}} 
\frac{\sigma_{\alpha}(n)\sigma_{\beta}(n-h)}{n^s}, \qquad \alpha, 
\beta, s \in \C, h \in \ZZ.$$ 
Somewhat surprisingly, the obstacle to applying the theory of 
\cite{DOS2} and \cite{B} towards our sought algebraic information is of 
analytic nature. The Rankin-Selberg method we need to apply (in 
\eqref{RS1}) cannot be completed in the 
form most commonly used. This is because of the non-uniform convergence 
of a series we must integrate against. We overcome this problem by 
replacing (in eq. \eqref{t=0}) the 
holomorphic 
Poincar\'e series required for our purposes
with a non-holomorphic Poincar\'e series.

The extra variable introduced in this way allows us to perform the 
analytic continuation of the shifted convolution obtained at the end of 
the computation (see \eqref{finfor}). Without this extra variable, the 
shifted convolution would not converge but now we can use the meromorphic 
continuation of \cite{K} to obtain a well-defined object. This enables us 
to prove the main
\begin{thm}\label{main} Let $f$ be a normalized weight $k$ cuspidal 
eigenform 
for SL$_2(\ZZ)$ and let $L^*_f(s)$ denote its completed L-function. 
Then, for each integer $r \ge 1$ we have
$$
\frac{L^*_f(k+r)}{L^*_f(k+1)} \in \mathcal D_r \mathcal D_{1}
\Q(f) \, \, \text{if $r$ is odd and} \quad
\frac{L^*_f(k+r)}{L^*_f(k+2)} \in \mathcal D_r \mathcal D_{2}
\Q(f) \, \,  \text{if $r$ is even.}
$$
Here $\mathcal D_m$ is the field generated over $\Q$ by 
$$\text{$\pi, i$ and
$D_l(k+m-2, -m-2; n), (n=k-2, \dots, k-2+m; l \in \mathbb N)$, if $m$ 
is odd, }$$ 
by
$$\text{$\pi, i$ and $D_l(k+m-1, -m-1; n), (n=k-1, \dots, k-1+m; l \in 
\mathbb 
N)$, if $m$ is even.}$$ 
Also, $\Q(f)$ is the field generated 
by the Fourier coefficients of $f.$ 
\end{thm}
It should 
be noted that, thanks to the 
explicit form of the analytic continuation in the region we need it, 
we even have an expression in terms of an Estermann zeta function.

It would be interesting to see if it is possible to determine the 
arithmetic nature of values of our shifted divisor sum Dirichlet series 
%(e.g. whether they are periods)
so that we obtain information about quotients of $L$-values.

%The resulting integrals have then to be evaluated outside the region of 
%initial convergence.
%
%\bf Acknowledgements \rm To be added later.

%\section{Observation}

\section{The field of L-values}
Let $f$ be a weight $k$ cuspidal eigenform 
 $$f(z)=\sum_{n=1}^{\infty}a(n) e^{2 \pi i n z}$$
  for SL$_2(\ZZ)$ normalised so that $a(1)=1$. Let $L_f(s)$ denote its 
L-function and consider its 
completed version
$$L^*_f(s)=(2 \pi)^{-s} \Gamma(s) L_f(s).$$
We will now define the double Eisenstein series introduced in 
\cite{DOS2}. Set $B=\{
\left ( \begin{smallmatrix} 1 & n \\ 0 & 1 \end{smallmatrix} \right )
; n \in \ZZ\}.$ Also set,
$$c_{\g}:=c \qquad \text{and $j(\g, z)=cz+d$} \qquad \text{for} \, \, \, 
\left ( \begin{smallmatrix} * & * \\ c & d \end{smallmatrix} \right ) \in 
\text{SL}_2(\ZZ).
$$
With the conversion $-\pi < \text{arg}(z) \le \pi$ define
\begin{equation}\label{dbleis3}
\EEE_{s, k-s}(z,w):= \sum_{\substack{\g, \delta \in B \backslash \G 
\\ c_{\g \delta^{-1}} > 0}} (c_{\g \delta^{-1}})^{w-1} 
\left(\frac{j(\g, z)}{j(\delta, z)} \right )^{-s} j(\delta, z)^{-k}.
\end{equation}
As shown in \cite{DOS2}, this series can be though of as a Rankin-Cohen 
bracket of not necessarily integer index, applied to pair of Eisenstein 
series. 
We also define the completed double Eisenstein series
\begin{equation}\label{complete}
\EEE^*_{s, k-s}(z,w):=\frac{ \G(s) \G(k-s) \G(k-w) \zeta(1-w+s)\zeta(1-w+k-s)}
{ e^{-s i \pi/2}2^{3-w}\pi^{k+1-w}   \G(k-1) } \EEE_{s, k-s}(z,w).
\end{equation}
In \cite{DOS2} it is proved: 
\begin{thm}\label{maineis} The  series $\EEE^*_{s, k-s}(z,w)$ 
converges absolutely and uniformly on compact sets for
which $2 < \Re (s) < k - 2$ and $\Re (w) < \Re(s)-1, k-1-\Re(s).$ It
has an 
analytic continuation to all $s,w \in \C$ and, as a function of $z$, it 
is a weight $k$ cusp form for $\G$. We have
\begin{equation}\label{lslw}
\langle \EEE^*_{s, k-s}(\cdot,w), f \rangle =  L^*_f(s) L^*_f(w)
\end{equation}
for any normalised cuspidal eigenform $f$ of weight $k.$
\end{thm}

With this theorem, we characterise the field of values of 
$L^*_f(s)$ using a method that is motivated by Zagier's technique 
(\cite{Z}). The latter is based on standard (positive integer index) 
Rankin-Cohen 
brackets. We will state and prove a slightly more general version of the 
characterisation given in \cite{DOS2}. 

For cusp forms $f_1, f_2, \dots$ we will be denoting by $\Q(f_1, f_2, 
\dots)$ the field obtained by adjoining to $\Q$ the Fourier coefficients 
of $f_1, f_2 \dots$. 
 \begin{prop}\label{Period} 
Let $f$ be a normalised cuspidal eigenform 
$f$ of weight $k$ and let $s_0 \in \C$ such that
$L_f^*(s_0) \ne 0.$ Then, for all $s, w \in \C$, with $L^*_f(s) \ne 0,$ 
$$\frac{L^*_f(w)}{L^*_f(s)} \in \Q(\EEE^*_{s_0, k-s_0}(\cdot, w), 
\EEE^*_{s_0, 
k-s_0}(\cdot, s), f).$$
\end{prop}
 \begin{proof}
With Th. \ref{maineis} we have
$$\frac{L^*_f(w)}{L^*_f(s)}=\frac{L^*_f(s_0)L^*_f(w)}{L^*_f(s_0)L^*_f(s)} 
=\frac{\langle \EEE^*_{s_0, k-s_0}(\cdot, s), f  \rangle}{\langle 
\EEE^*_{s_0, k-s_0}(\cdot, w), f \rangle}
=\frac{\langle \EEE^*_{s_0, k-s_0}(\cdot, s), f  \rangle/\langle 
f, f \rangle}{\langle 
\EEE^*_{s_0, k-s_0}(\cdot, w), f \rangle /\langle f, f\rangle}.
$$
By a general result (see, e.g. \cite{S}, Lemma 4), the numerator of the 
last fraction belongs to $\Q(\EEE^*_{s_0, k-s_0}(\cdot, s), f)$. Likewise 
for 
the denominator. This implies the result.
\end{proof}
We state a result which can be deduced from this proposition.

According to 
Manin's Periods Theorem \cite{M} there are $\omega^+(f), \omega_-(f) \in 
\R$ such that
$$L^*_f(s)/\omega^+(f), \qquad L^*_f(w)/\omega^-(f) \in \Q(f)$$
for all $s, w$ with $1 \le s, w \le k-1$ and $s$ even, $w$ odd. 
Although we know that $\omega^{\pm}(f)$ are periods (cf. Sect. 3.4 of 
\cite{KZ}) and that (when appropriately normalised) their product is 
$\langle f, f \rangle$, little is known about their quotient.
However, 
Prop. \ref{Period} 
implies the following characterisation of the field to which their 
quotient belongs.
 \begin{prop} Let $f$ be a normalised cuspidal 
eigenform of weight $k$ for SL$_2(\ZZ)$ such that $L^*_f(k/2) \ne 0$. (In 
particular, $k \equiv 0 \mod 4.$ ) Then
$$\frac{\omega^+(f)}{\omega^-(f)} \in \Q(\pi, i,  \EEE^*_{k/2, 
k/2}(\cdot, 
4), f).$$
%This implies that, when $f$ has rational Fourier coefficients, 
%$\frac{\omega^+(f)}{\omega^-(f)}$ 
%is a period in the sense of Kontsevich-Zagier.
\end{prop}
\begin{proof}
Set $w=4$, $s=3$ and $s_0=k/2$ in Prop. \ref{Period}. Since 
$L^*_f(4)/\omega^+(f), L^*_f(3)/\omega^-(f) \in \Q(f)$,
we deduce
$$\frac{\omega^+(f)}{\omega^-(f)} \in \Q(
\EEE^*_{k/2, k/2}(\cdot, 3), \EEE^*_{k/2, k/2}(\cdot, 4), f).$$
With Prop. 2.4 of \cite{DOS2}, 
$$\EEE^*_{k/2, k/2}(\cdot, 3)=\EEE^*_{\left( \frac{k}{2}-2 \right )+2, 
\left( \frac{k}{2}-2 \right )+2}(\cdot, 2+1)$$ equals (up to an factor in 
$\Q(\pi, i)$) the Rankin-Cohen 
bracket 
\begin{equation*} \label{rco}
[E_{k/2},E_{k/2}]_2 
:= \sum_{r=0}^2 (-1)^r \binom{k_1+1}{2-r}
\binom{k_2+1}{r}E_{k/2}^{(r)}E_{k/2}^{(2-r)}
\end{equation*}
where 
$$E_{2m}(z)=\sum_{\g \in \Gamma_{\infty} \backslash
\G} j(\gamma, z)^{-2m}=
-\frac{B_{2m}}{4m}
+\sum_{n=1}^{\infty} 
\sigma_{2m-1}(n)e^{2 \pi i n z}.$$
From the general theory of Rankin-Cohen brackets, $[E_{k/2},E_{k/2}]_2$ 
 is a cusp from, and 
it is clear that it has
rational Fourier coefficients. 

%To prove the last implication of the proposition, we first note that, 
%from Prop. \ref{Period}, we have
%$$\EEE^*_{k/2, k/2}(s, 4) =    \sum L^*(f,k/2) L^*(f, 4) 
%\frac{f(z)}{\langle f, f, \rangle}$$
%the Fourier coefficients of  
%$\EEE^*_{k/2, k/2}(\cdot, 4)$ are $\mathbb Q$-linear combinations of 
%products of $L^*_f(4)$ and $L^*_f(k/2)$ for $f$ ranging over an 
%orthonormal 
%basis with rational Fourier coefficients. Since these are critical 
%L-values, they are periods in the sense of Kontsevich-Zagier (see 
%Section 3.4 of \cite{KZ}). From this we deduce the last implication.
  \end{proof}
{\it Remark.} In \cite{PP} (5.13), it is shown that, in a different 
setting (odd 
weight $k$ and higher level of the group), this quotient does have a 
simple explicit expression. 

\section{Brown's kernel} \label{FB} 
We maintain the notation of the previous section.
In \cite{B}, F. Brown gives a kernel for a certain product of values of 
$L_f(s)$.

Let $i, j \ge 0$ be integers and let $s$ be such that $i+j+2 \Re(s)>2.$  
For $z=x+iy$ set
$$\mathcal E_{i, j}^s(z)=\frac{1}{2}\sum_{\g \in B \backslash
\G} \frac{y^s}{j(\gamma, z)^{i+s} j(\gamma, \bar z)^{j+s}}.
$$
This series converges absolutely in the indicated region and satisfies
\begin{equation} \label{trlaw} 
\mathcal E_{i, j}^s(\g z)=j(\g, z)^i j(\g, \bar z)^j \mathcal E_{i, 
j}^s(z) \qquad \text{for all $\g \in \G$ and $z \in \HH$}.
\end{equation}
It further has a meromorphic continuation to the entire complex plane
since 
\begin{equation} \mathcal E_{i, j}^s(z)=y^{-\frac{i+j}{2}}E_{i-j}(s+
\frac{i+j}{2})(z)
\end{equation}  
where $E_m(s)(z)$ stands for the weight $m$ Eisenstein series
$$E_m(s)(z)=\sum_{\g \in \Gamma_{\infty} \backslash
\G} \text{Im}(\g z)^s \left( \frac{|j(\gamma, z)|}{j(\gamma, 
z)} \right )^{m}.$$
Here $\Gamma_{\infty}$ is the stabiliser of $\infty$.
 
For $m \in \ZZ_+$ let $\langle \cdot, \cdot \rangle_m$ be the pairing on 
real-analytic functions whose product vanishes exponentially at $\infty$ 
which is 
given by the formula of the Petersson scalar product in weight $m$: 
$$\langle g, h \rangle_m= \int_{\mathfrak F}g(z) \overline{h(z)} 
y^{m}\frac{dx dy}{y^2}$$
where $\mathfrak F$ is a fundamental domain of SL$_2(\ZZ)$. Notice that 
we do not require $g, h$ to be $\Gamma$-invariant as in the case of 
the actual Petersson scalar product.

Corollary 11.12 of \cite{B} implies:
\begin{thm}
\label{corB} 
Let $r \ge 1$ be an integer and let $a, b \ge 2$ be integers 
such that 
\begin{equation}\label{k}
k=2a+2b-2r-2.
\end{equation}
Then for each normalised cuspidal eigenform $f$ of weight k and for 
each $s \in \C$ , we have
\begin{multline}
\label{kernelB}
\frac{\pi^{2a-s-k-r}}{2^{1-2a}}\Gamma(s+k+r-2a) \zeta(2s+k+2r-2a) \langle
E_{2a} \mathcal E_{k+r-2a, r}^s, f \rangle_{k+r}  \\
=L^*_f(s+k+r-1) L^*_f(s+k+r-2a).
\end{multline}
\end{thm}
Theorem \ref{corB} combined with Theorem \ref{maineis} implies that 
$\EEE^*_{s+k+r-1, -s-r+1}(z, s+k+r-2a)$
is a holomorphic projection of 
$$\frac{\pi^{2a-s-k-r}}{2^{1-2a}}\Gamma(s+k+r-2a) \zeta(2s+k+2r-2a) 
E_{2a} \mathcal E_{k+r-2a, r}^s$$ in the sense that
$\EEE^*_{s+k+r-1, -s-r+1}(z, s+k+r-2a)$ is a holomorphic cusp form and
for all weight $k$ forms $f$ we have
\begin{multline}
\label{holpr}
\frac{\pi^{2a-s-k-r}}{2^{1-2a}}\Gamma(s+k+r-2a) \zeta(2s+k+2r-2a) \langle
E_{2a} \mathcal E_{k+r-2a, r}^s, f \rangle_{k+r}  \\
=\langle \EEE^*_{s+k+r-1, -s-r+1}(z, s+k+r-2a), f \rangle_{k}.
\end{multline}
\section{Fourier coefficients}
In this section we will compute the Fourier coefficients of  
$\EEE^*_{s+k+r-1, -s-r+1}(z, s+k+r-2a)$ so that we can apply Prop. 
\ref{Period} to deduce Theorem \ref{main}.

By the formula for Fourier coefficients of a cusp form in terms of inner 
products against Poincar\'e series, \eqref{holpr} implies that 
\begin{multline}
\label{Fexp}
\text{the $l$-th Fourier coefficient of $\EEE^*_{s+k+r-1, -s-r+1}(z, 
s+k+r-2a)=$} \\
\frac{\pi^{2a-s-k-r} (4 \pi l)^{k-1}}{2^{1-2a}\Gamma(k-1)}
\Gamma(s+k+r-2a) \zeta(2s+k+2r-2a)
     \langle 
E_{2a} \mathcal E_{k+r-2a, r}^s, P_l\rangle_{k+r}
\end{multline}
where
$$P_l(z)=\sum_{\g \in \Gamma_{\infty} \backslash \G} \frac{e^{2 \pi i l \g z}}{j(\g, z)^k}$$
is the $l$-th Poincar\'e series of weight $k$.

Therefore, the determination of the field of L-values appearing in 
Theorem \ref{Period}
reduces to the computation of the inner products
\begin{equation}\label{RS1}
\langle E_{2a} \mathcal E_{k+r-2a, r}^s, P_l\rangle_{k+r}. 
\end{equation}
We will compute this integral using the Rankin-Selberg method. However, 
we have to modify slightly the method because otherwise, along the way 
we obtain a series which cannot be interchaged with integration.
To overcome this difficulty we, 
essentially, regularise the
integral. Let 
$$P_l^t(z):=\sum_{\g \in \Gamma_{\infty} \backslash \G} 
\Im(\g z)^{\bar t}\frac{e^{2 \pi i l \g z}}{j(\g, z)^k}$$
is the $l$-th non-holomorphic Poincar\'e series of weight $k$. It is 
clear that
\begin{equation} \label{t=0}
\langle E_{2a} \mathcal E_{k+r-2a, r}^s, P_l\rangle_{k+r}=
\langle E_{2a} \mathcal E_{k+r-2a, r}^s, P_l^t\rangle_{k+r}|_{t=0}.  
\end{equation}
Let $\Re (t) \gg 0$. With \eqref{trlaw} we now have
\begin{multline}
\langle E_{2a} \mathcal E_{k+r-2a, r}^s, P_l^t\rangle_{k+r}=
\int_{\mathfrak F} \sum_{\g \in \Gamma_{\infty} \backslash \G} \frac{E_{2a} (\g z)}{j(\g, z)^{2a}}
\frac{ \mathcal E_{k+r-2a, r}^s(\g z)}{j(\g, z)^{k+r-2a} j(\g, \bar z)^r}
 \Im(\g z)^{ t} \frac{\overline{e^{2 \pi i l \g z}}}{j(\g, \bar z)^k} y^{k+r} d \mu z\\
=\int_{\mathfrak F} \sum_{\g \in \Gamma_{\infty} \backslash \G} E_{2a} (\g z)
\mathcal E_{k+r-2a, r}^s(\g z) \Im(\g z)^{t+k+r} 
\overline{e^{2 \pi i l \g z}} d \mu \g z.
\end{multline}
Since $E_{2a} \mathcal E_{k+r-2a, r}^s$ has polynomial growth at the 
cusps, we have 
uniform convergence when $\Re t$ is large enough and we can complete the unfolding as usual. Then
the last integral becomes 
\begin{equation}\label{RS}
\int_0^{\infty}  \int_0^1 E_{2a} (z)
\mathcal E_{k+r-2a, r}^s(z) y^{t+k+r-2} e^{-2 \pi l y} e^{-2 \pi i l x} dx dy.
\end{equation}
To complete the computation we need the Fourier expansion of $E_m(s)(z)$ 
as given 
in Prop. 11.2.16 of \cite{CS} (or Th. 3.1 of \cite{DOS1}), for {\it 
integer $s$} with $s+r-a+k/2>0$ which will be the case that interests us. 
It gives
\begin{multline}
\mathcal E^s_{k+r-2a, r}(z)= y^{a-r-\frac{k}{2}} 
E_{k-2a}(s+r-a+\frac{k}{2})(z) =\\
y^{a-r-\frac{k}{2}}a_0(y)+\sum_{n \ge 1}\left ( 
\frac{\sigma_{2s+k+2r-2a-1}(n)}{n^{s+r-a+\frac{k}{2}}} 
\left ( \sum_{a-\frac{k}{2} \le j \le s-1+\frac{k}{2}+r-a} \alpha_j^+ (4 
\pi n y)^{-j}\right ) e^{-2 \pi n y} 
y^{a-r-\frac{k}{2}} \right )e^{2 \pi i nx} \\
+\sum_{n \le -1}
\left ( \frac{\sigma_{2s+k+2r-2a-1}(|n|)}{|n|^{s+r-a+\frac{k}{2}}} \left ( 
\sum_{\frac{k}{2}-a \le j \le s-1+\frac{k}{2}+r-a} \alpha_j^- (4 \pi |n| y)^{-j}\right ) e^{2 \pi n y} 
y^{a-r-\frac{k}{2}} \right )e^{2 \pi i nx}. 
\end{multline}
Here $$a_0(y)=\left | \frac{k}{2}-a\right |! \Big ( 
\binom{s+r-a+\frac{k}{2}-1+\left | \frac{k}{2}-a\right |}{\left | 
\frac{k}{2}-a\right |}  
\Lambda(2s+2r-2a+k)y^{s+r-a+\frac{k}{2}}+$$
$$\binom{\left | \frac{k}{2}-a\right | 
-s-r+a-\frac{k}{2}}{\left | 
\frac{k}{2}-a\right |}\Lambda(2-2s-2r+2a-k)y^{1-s-r+a-\frac{k}{2}} 
\Big ),$$ where $\Lambda(s):=\pi^{-s/2} \Gamma(s/2) \zeta(s) $ and 
$$\alpha_j^{\pm}:=(-1)^j(j+\frac{|k-2a|}{2})! \binom{s+r-a+\frac{k}{2}-1+\frac{|k-2a|}{2}}{j+\frac{|k-2a|}{2}} 
\binom{\pm \frac{k-2a}{2}-s-r+a-\frac{k}{2}}{j \pm \frac{k-2a}{2}}.$$
Here $\binom{a}{b}$ with $a<0$ are defined in accordance to the 
convention that, if $j \ge 0$, then $$\binom{-a+j-1}{j}=(-1)^j 
\binom{a}{j}.$$

Therefore the coefficient of $e^{2 \pi i l x}$ ($l>0$) in the Fourier 
expansion of $E_{2a}(z) \mathcal E^s_{k+r-2a, r}(z)$ is
$$\sum_{n \le l} b_n(y) \sigma_{2a-1}(l-n) e^{-2 \pi (l-n)y}$$
where, if $n \ne 0$,
$$b_n(y):= \frac{\sigma_{2s+k+2r-2a-1}(|n|)}{|n|^{s+r-a+\frac{k}{2}}} 
\left ( 
\sum_{\sgn(n)(a-\frac{k}{2}) \le j \le s-1+\frac{k}{2}+r-a} 
\alpha_j^{\sgn(n)} (4 \pi |n| y)^{-j}\right ) e^{-2 \pi |n| y}
y^{a-r-\frac{k}{2}} \quad \text{and}$$
$$b_0(y):=y^{a-r-k/2}a_0(y).$$
Also, for convenience, we set $\sigma_{2a-1}(0):=-B_{2a}/(4a)$.

Therefore, with \eqref{RS} we have that 
for $\Re t \gg 0$, 
\begin{equation}\label{Fourier}
\langle E_{2a} \mathcal E_{k+r-2a, r}^s, P_l^t\rangle_{k+r}=
\int_0^{\infty} e^{-2 \pi ly} \left (\sum_{n \le l} b_n(y) 
\sigma_{2a-1}(l-n) e^{-2 \pi (l-n)y} \right ) y^{t+k+r-2} dy.
\end{equation}
Since the finitely many terms corresponding to $b_n(y)$ with $n \ge 0$ 
can be directly evaluated at $t=0$ to give
rational linear combinations of powers of $\pi$ and $i$, we will  focus 
on the 
infinitely many terms indexed by $n<0$. \begin{multline}\label{negterms}
\int_0^{\infty} e^{-2 \pi ly} \left (\sum_{n <0} b_n(y) 
\sigma_{2a-1}(l-n) e^{-2 \pi (l-n)y} \right ) 
y^{t+k+r-2} dy \\
=\sum_{\frac{k}{2}-a \le j \le s-1+\frac{k}{2}+r-a} \alpha_j^{-} (4 \pi)^{-j}
\sum_{n <0}  \frac{\sigma_{2s+k+2r-2a-1}(|n|) \sigma_{2a-1}(l-n)}{|n|^{s+r-a+\frac{k}{2}+j}} 
\int_0^{\infty} e^{-4 \pi (l-n)y}y^{\frac{k}{2}+a-j-1+t}\frac{dy}{y} 
\\
=\sum_{\frac{k}{2}-a \le j \le s-1+\frac{k}{2}+r-a} \alpha_j^{-} (4 \pi)^{1-\frac{k}{2}-a-t}
\Gamma(\frac{k}{2}+a-j-1+t)
\sum_{n >0}  \frac{\sigma_{2s+k+2r-2a-1}(n) \sigma_{2a-1}(l+n)}{n^{s+r-a+\frac{k}{2}+j} 
(n+l)^{\frac{k}{2}+a-j-1+t}} 
\end{multline}
%Now assume that $2s \ge r$. 
Since $ j \le s-1+\frac{k}{2}+r-a$ we can expand binomially the term
$((n+l)-l)^{s+\frac{k}{2}+r-a-1-j}$. This, together with 
the trivial identity $\sigma_w(n)=n^w \sigma_{-w}(n)$, imply that 
the last sum of \eqref{negterms} becomes
\begin{multline} \label{finfor}
\sum_{n > 0}  \frac{\sigma_{-2s-k-2r+2a+1}(n) 
n^{s+\frac{k}{2}+r-a-1-j}\sigma_{2a-1}(l+n)}
{(n+l)^{\frac{k}{2}+a-j-1+t}} =
\sum_{\mu=0}^{s+\frac{k}{2}+r-a-1-j} \binom{s+\frac{k}{2}+r-a-1-j}{\mu} 
\times \\
(-l)^{s+\frac{k}{2}+r-a-1-j-\mu} D_l(2a-1, -2s-k-2r+2a+1; 
\frac{k}{2}+a-j-1+ t-\mu)
\end{multline}
where
$$D_l(\alpha, \beta; w):= \sum_{n > l}  \frac{\sigma_{\alpha}(n) 
\sigma_{\beta}(n-l)} {n^{w}}.$$
In view of \eqref{t=0}, we aim to investigate the series in 
\eqref{finfor} at $t=0$. We are mainly interested in the case that least 
one of $s+k+r-1$ and $s+k+r-2a$ is outside the critical strip. In this
case, the series $D_l$ appearing in \eqref{finfor} are not in the initial 
region of 
convergence when $t=0.$ However, this convolution series has been studied 
by M. K{\i}ral in \cite{K} and has given the analytic continuation. 
Here we do not need the full analytic continuation established in 
\cite{K} but rather a partial extension. 
%We will summarise it in the 
%case we need it making at the same time the adjustments required in our 
%setting. 
Since \cite{K} has not appeared in print yet, we give here a proof of the 
part of the analytic continuation we need for our purposes.

For $r$ odd, set 
$$s=1, \qquad \qquad a=\frac{k+r-1}{2}.$$
These values of $s$ and $a$ satisfy the 
conditions $s+r-a+k/2>0$ and 
$a, b=(k-2a+2r+2)/2 \ge 2$ that are required for our construction. 
Further, with these values of $s, a$,
 the parameter $j$ in 
\eqref{negterms} ranges between $(1-r)/2$ and $(1+r)/2$.
For each of these values of $j$, the parameter $\mu$ in 
\eqref{finfor} ranges between $0$ and $\frac{1+r}{2}-j$ and thus
$j+\mu$ ranges between
$(1-r)/2$ and $(1+r)/2.$

Therefore, we need to show all shifted convolutions
$$D_l(k+r-2, -r-2, \nu+t); \qquad \qquad (\nu=k-2, \dots, k-2+r)$$
appearing in \eqref{finfor} have an analytic continuation in a 
neighbourhood of $t=0$.

To achieve that, we first note that Prop. 6 of \cite{K} can be adjusted 
to say that, in the region of absolute convergence, we have
\begin{equation} \label{esterm}
D_l(k+r-2, -r-2; \nu+t)=\zeta(3+r) \sum_{m=1}^{\infty} 
m^{-3-r} \sum_{\substack{x \mod m \\ (x, m)=1}} e^{\frac{-2 \pi i x 
l}{m}}
E_l(\nu+t, k+r-2; \frac{x}{m})
\end{equation}
where
$E_l(s, \alpha; \frac{x}{m})$
the (truncated) Estermann zeta function, defined, for $\Re(s) \gg 0$, by
$$E_l(s, \alpha; \frac{x}{m})
:=\sum_{n > l} 
\frac{\sigma_{\alpha}(n) e^{2\pi i \frac{xn}{m}}}{n^s} \qquad s, \alpha 
\in \C, x/m \in \Q.$$
The (full) Estermann zeta function $E(s, \alpha; \frac{x}{m}):=
E_0(s, \alpha; \frac{x}{m})$ has a meromorphic continuation through its 
expression in terms of Hurwitz zeta functions $\zeta(s, x)$. This 
expression, in our 
case, can be stated, for $\Re(t) \gg 0$ as:
$$E(\nu +t, k+r-2; \frac{x}{m})=m^{k+r-2-2\nu-2t}\sum_{u, v=1}^m
e^{2\pi i \frac{xuv}{m}} 
\zeta(-k-r+2+\nu+t, \frac{u}{m})
\zeta(\nu+t, \frac{v}{m}).
$$
It is well-known that $\zeta(s, x)$ has a meromorphic continuation to 
$\mathbb C$, with only a simple pole at $s=1$. Therefore the 
RHS is analytic in a neighbourhood of $t=0$, since $\nu \in \{k-2, \dots, 
k-2+r\}$, thus giving the analytic continuation the LHS in this 
neighbourhood. Further, when $x \in (0, 1),$ $\zeta(\nu+t, x)$ is bounded 
by a constant independent of $x$, since $\nu+\Re(t)>1$. 
By the functional equation of 
$\zeta(s, x)$ (or, in the case $\nu=k+r-2$, the Taylor expansion of 
$\zeta(s, x+1)$ at $x=0$), we deduce that, 
for $t$ in a small neighbourhood of $0$
$$\zeta(-k-r+2+\nu+t, \frac{u}{m}) \ll_{k, r} m^{t}+1$$
Therefore, 
$$E(\nu +t, k+r-2; \frac{x}{m}) \ll_{k, r} 
m^{k+r-2\nu-t}+m^{k+r-2\nu-2t}.$$
Applying this bound to the series of the RHS of \eqref{esterm}, we see 
that each term is holomorphic in a neighbourhood of $t=0$ and bounded by 
$m^{k-2\nu-t-1}+m^{k-2\nu-2t-1}$. Therefore, the function 
$D_{l}(k+r-2, -r-2, \nu+t)$ 
is 
holomorphic at 
$t=0$ for all $\nu \in \{k-2, \dots, k-2+r\}$.

For $r$ even, set 
$$s=1, \qquad \qquad a=\frac{k+r}{2}.$$
Working exactly in the same way as above, we find that 
$D_{l}(k+r-1, -r-1, \nu+t)$ 
is holomorphic at 
$t=0$ for all $\nu \in \{k-1, \dots, k-1+r\}$.

\medskip

{\it Proof of Theorem \ref{main}:} 
As mentioned above, the terms corresponding to non-negative $n$ in the 
sum in \eqref{Fourier} lead to elements of $\Q(\pi, i)$. 

Let $r$ be odd. What we just 
proved, together with 
\eqref{Fexp} then imply that 
the $l$-th Fourier coefficient 
of 
$\EEE^*_{k+r, -r}(z, 2)$ 
belongs to the field 
$$
\mathcal D_r=\Q(\pi, i, D_l(k+r-2, -r-2; k-2), 
\dots, D_l(k+r-2, -r-2; k-2+r))
$$
obtained by adjoining to $\Q(\pi, i)$ the values of the shifted 
convolution $D_l(k+r-2, -r-2; n)$ at $n=k-2, \dots k-2+r.$

On the other hand, \eqref{lslw} implies that 
$$\EEE^*_{k+r, -r}(z, 2)=\EEE^*_{2, k-2}(z, k+r).$$
Therefore, with Proposition \ref{Period} for $s_0=2$ and odd $r' \ge 1$ 
$$\frac{L^*_f(k+r)}{L^*_f(k+r')} \in \mathcal D_r \mathcal D_{r'}
\Q(f)$$
Setting $r'=1$, we deduce the first case of Theorem \ref{main}.

Let $r$ be even. In the same way as above we find that
the $l$-th Fourier coefficient of $\EEE^*_{k+r, -r}(z, 1)=
\EEE^*_{1, k-1}(z, k+r)$ belongs to the field 
$$
\Q(\pi, i, D_l(k+r-1, -r-1; k-1), D_l(k+r-1, -r-1; k), \dots, 
D_l(k+r-1, -r-1; k-1+r))
$$
Applying Proposition \ref{Period} with $s_0=1$, we deduce, for even 
$r'>1$, 
$$\frac{L^*_f(k+r)}{L^*_f(k+r')} \in \mathcal D_r \mathcal D_{r'}
\Q(f)$$
where $\mathcal D_r$ denotes the field 
obtained by adjoining to $\Q(\pi, i)$ the values of the shifted 
convolution $D_l(k+r-1, -r-1; n)$ at $n=k-1, \dots k-1+r.$
Setting $r'=2$, we deduce the second case of Theorem \ref{main}.
 \qed


\begin{thebibliography}{5}   

\bibitem{B} F. Brown {\it Multiple Modular Values for SL$_2(\ZZ)$}
 arXiv:1407.5167

\bibitem{CS} H. Cohen, F. Str\"omberg {\it
Modular Forms: A Classical Approach} (to appear)

\bibitem{DOS1} N. Diamantis and C. O'Sullivan {\it 
Kernels of L-functions of cusp forms} 
Math. Annalen, 346, (2010), no. 4, 897 - 929

\bibitem{DOS2} N. Diamantis and C. O'Sullivan {\it Kernels for products 
of L-functions} 
Algebra and Number Theory, 7 (2013), no. 8, 1883 - 1917.

\bibitem{K} M. K{\i}ral {\it Shifted Divisor Sum Dirichlet Series}
preprint

\bibitem{M} 
Y. Manin {\it Periods of cusp forms, and $p$-adic Hecke 
series.} Mat. Sb. (N.S.), 21(134):371-393 (1973)

\bibitem{M1} 
Y. Manin {\it Iterated integrals of modular forms and non-commutative 
modular symbols} Algebraic
geometry and number theory, 565--597, Prog. Math. 253 (2006)

\bibitem{M2} 
Y. Manin {\it Iterated Shimura integrals} Moscow Math. J. 5 (2005), 
869--881

\bibitem{S} G. Shimura {\it The special values of the zeta functions 
associated with cusp forms} Comm. Pure Appl. Math.,
29(6):783-804 (1976)

\bibitem{KZ} M. Kontsevich, D. Zagier {\it 
Periods.} Mathematics Unlimited--2001 and Beyond (B. Engquist and W. Schmid, 
eds.), Springer, Berlin-Heidelberg-New York (2001), 771--808

\bibitem{PP} V. Pasol, A. Popa {\it Modular forms and period 
polynomials.} Proc. Lond. Math. Soc. 107/4 (2013), 713--743

\bibitem{Z} D. Zagier {\it 
Modular forms whose Fourier coefficients involve zeta-functions of 
quadratic fields}
in Modular Functions of One Variable VI, Lecture Notes in Math. 627, 
Springer-Verlag, Berlin-Heidelberg-New York (1977) 105--169

\iffalse 
\bibitem{GR} I.S. Gradshteyn and I.M. Ryzhik {\it Table of Integrals,
Series and Products,} fifth edition, Academic Press Inc.


\bibitem{H} J. Hoffstein {\it Multiple Dirichlet Series and Shifted
Convolutions} arXiv:1110.4868v1

\bibitem{H1} J. Hoffstein, T. A. Hulse, A. Reznikov {\it Multiple Dirichlet Series and Shifted
Convolutions} J. Number Theory (accepted) arXiv:1110.4868v2

\bibitem{HL} J. Hoffstein, M. Lee {\it  Shifted Multiple Dirichlet Series} arXiv:1412.5917

\bibitem{I} H. Iwaniec {\it Topics in classical automorphic forms}
Graduate Studies in Mathematics, Vol. 17, AMS, 1991

\bibitem{KS} H. Kim {\it Functoriality for the exterior square of GL$_4$ and the symmetric fourth of GL$_2$.
With appendix 1 by Dinakar Ramakrishnan and appendix 2 by Kim and Peter Sarnak. }
 J. Amer. Math. Soc. 16 (2003), no. 1, 139--183.

\bibitem{KMV} E. Kowalski, P. Michel, P., J. VanderKam {\it
Rankin-Selberg L-functions in the level aspect.} Duke Math. J.
114
(2002), no. 1, 123--191

\bibitem{LLY} Y.-K. Lau, J. Liu, Y. Ye {\it A new bound
$k^{2/3+\epsilon}$ for Rankin-Selberg $L$-functions for Hecke
congruence subgroups} IMRP Int. Math. Res. Pap. 2006:7 (2006) 1--78.

\bibitem{M} T. Miyake {\it Modular forms} Springer, 2003

\bibitem{PR} Y. Petridis, M. Risager {\it Modular symbols have a
normal distribution} Geom. Funct. Anal. 14 (2004), no. 5, 1013--1043.

\bibitem{Ru} W. Rudin {\it Real and complex analysis}
Third Edition (1987), McGraw-Hill

\bibitem{S1} P. Sarnak {\it Integrals of products of eigenfunctions}
Internat. Math. Res. Notices 6
(1994), 251--260.

\bibitem{S} P. Sarnak {\it Estimates for Rankin-Selberg $L$-functions
and quantum unique ergodicity} J. Funct. Anal. 184
(2001) 419--453.
\fi

\end{thebibliography}
\end{document}